\documentclass[a4paper]{amsart}

\usepackage{latexsym, amssymb,amsmath, amsthm,amsfonts}

\newcommand{\FF}{\mathbb{F}}

\newcommand{\ZZ}{\mathbb{Z}}
\newcommand{\kk}{\Bbbk}
\newcommand{\kv}{{\kk[V]}}

\newcommand{\kvg}{{\kk[V]^{G}}}

\def\sep{\operatorname{sep}}

\newcommand{\NGV}{{\mathcal{N}_{G,V}}}

\def\SL{\operatorname{SL}}

\def\reg{\operatorname{reg}}
\def\SL2{\operatorname{SL}_{2}(K)}

\def\GL2{\operatorname{GL}_{2}(K)}
\def\Ga{{\mathbb G}_{a}}

\def\INVSL2{$K[V]^{operatorname{SL}_{2}(K)}$}
\def\INVSO2{$K[V]^{operatorname{SO}_{2}(K)}$}
\def\INVGL2{$K[V]^{operatorname{GL}_{2}(K)}$}

\def\Hom{\operatorname{Hom}}
\def\GL{\operatorname{GL}}
\def\SL{\operatorname{SL}}
\def\id{\operatorname{id}}

\def\Stab{\operatorname{Stab}}

\def\Z{\mathbb{Z}}
\def\N{\mathbb{N}}

\newtheorem{Lemma}{Lemma}[section]
\newtheorem{Theorem}[Lemma]{Theorem}
\newtheorem{Corollary}[Lemma]{Corollary}

\newtheorem{Prop}[Lemma]{Proposition}
\newtheorem{Conj}[Lemma]{Conjecture}
\newtheorem*{Corollary of Conjecture}{Corollary of Conjecture}

\theoremstyle{definition}

\theoremstyle{remark}
  \newtheorem{rem}[Lemma]{Remark}
\newtheorem{eg}[Lemma]{Example}

\newtheoremstyle{Acknowledgments}
  {}
    {}
     {}
     {}
    {\bfseries}
    {}
     {.5em}
     {\thmname{#1}\thmnumber{ }\thmnote{ (#3)}}
\theoremstyle{Acknowledgments}
\newtheorem{ack}{Acknowledgments.}

\title[Zero-separating invariants]
{Zero-separating invariants\\ for linear algebraic groups}

\author{Jonathan Elmer}
\address{University of Aberdeen\\
King's College, Aberdeen\\
AB24 3UE}
\email{j.elmer@abdn.ac.uk}

\author{Martin Kohls}
\address{Technische Universit\"at M\"unchen \\
 Zentrum Mathematik-M11\\
Boltzmannstrasse 3\\
 85748 Garching, Germany}
\email{kohls@ma.tum.de}

\date{\today}

\begin{document}
\maketitle

\begin{abstract}
Let $G$ be linear algebraic group over an algebraically closed field~$\kk$
acting rationally on a $G$-module $V$, and $\NGV$ its nullcone. Let $\delta(G,V)$ and~$\sigma(G,V)$ denote the minimal number $d$, such that for any
$v\in V^{G}\setminus\NGV$ and $v\in V\setminus\NGV$ respectively, there exists a
 homogeneous invariant $f$ of positive degree at most $d$ such that $f(v)\ne
0$. Then $\delta(G)$ and $\sigma(G)$ denote the supremum of these numbers taken over all $G$-modules $V$. For positive characteristics, we show that $\delta(G)=\infty$ for any subgroup
$G$ of $\GL_{2}(\kk)$ which contains an infinite unipotent group, and $\sigma(G)$
is finite if and only if $G$ is finite. In characteristic zero, $\delta(G)=1$
for any group $G$, and we show that if~$\sigma(G)$ is finite, then $G^{0}$ is
 unipotent. Our results also lead to a more elementary proof that
$\beta_{\sep}(G)$ is finite if and only if $G$ is finite.
\end{abstract}

\section{Introduction}\label{SecIntro}
In invariant theory, the notion of geometric reductivity is of great
importance. It implies finite generation of the invariants, the separabability
of disjoint orbit closures by invariants, and in characteristic zero even algebraic
properties like the Cohen-Macaulay\-ness of the invariant ring.
It is defined to be the property that any non-zero fixed point of a
finite dimensional rational representation can be separated from zero by a homogeneous invariant
of positive degree. Similarly, by definition any point outside the nullcone
can be separated from zero by a homogeneous positive degree invariant. It is a
natural question to ask what is the maximum degree needed for a given
representation. While in our recent paper \cite{ElmerKohlsSigmaDelta} we gave
some (partial) answers to these questions for the case of finite groups, the
current paper concentrates on the case of infinite groups. Before we go into
more details, we fix our setup.

Let $G$ be a linear algebraic group over an algebraically closed field $\kk$,
$V$ a finite dimensional rational representation of $G$ (which we will call a
$G$-module), and denote by $\kv\cong S(V^{*})$ the ring of polynomial
functions $V \rightarrow \kk$. The action of $G$ on $V$ induces an action of
$G$ on $\kv$ via $(g\cdot f)(v) := f(g^{-1}v)$ for $g\in G$,  $f\in\kv$ and
$v\in V$. The set of $G$-invariant polynomial functions under this action is
denoted by $\kvg$, and inherits a natural grading from $\kv$, since the given action is degree-preserving. We denote by
$\kv^{G}_{d}$ the set of polynomial invariants of degree $d$ and the zero-polynomial, and by
$\kv^{G}_{\leq d}$ the set of polynomial invariants of degree at most
$d$. For any subset $S$ of $\kv$, we define $S_{+}$ as the set of elements in $S$
with constant term zero. Then $\NGV:=\mathcal{V}(\kk[V]^{G}_{+})$ denotes the
nullcone of $V$.
A linear algebraic group is said to be geometrically reductive, if
for any
$G$-module $V$, we have $V^{G}\cap \NGV=\{0\}$, i.e. for all nonzero $v \in V^G$, there exists  $f \in
\kk[V]_{+}^G$ such that $f(v) \neq 0$. This inspires the definition
of a $\delta$-set: for a linear algebraic group $G$, let us say  a subset $S \subseteq \kvg$ is a
$\delta$-\emph{set} if, for all $v \in V^G\setminus \NGV$, there
exists an $f \in S_{+}$ such that $f(v) \neq 0$. We shall call a subalgebra of $\kk[V]^G$ a $\delta$-\emph{subalgebra} if it is a $\delta$-set. The quantity $\delta(G,V)$ is then defined as
\[
\delta(G,V) = \min\{d\ge 0|\quad \kk[V]^G_{\leq d} \, \text{ is a
  $\delta$-set} \,  \}.
\] 
Define further 
$$\delta(G):= \sup\{\delta(G,V) |\quad V \,\text{ a
   $G$-module}\},$$ where we take the supremum of an unbounded set to be infinity.
A reductive group is called linearly reductive if $\delta(G)=1$. Over a field of characteristic zero, Nagata and
Miyata \cite{NagataMiyata} have shown that reductive groups are
linearly reductive. In fact their proof shows that in characteristic zero,
for any linear algebraic group $G$ and any $G$-module $V$, $\delta(G,V)$ equals $1$ or
$0$ (the latter being the case when $V^{G}\subseteq \NGV$), see Proposition \ref{NagataMiyata}.
A natural, but seemingly neglected question, is: for which geometrically
reductive groups $G$ is $\delta(G)$ strictly greater than $1$, but still
finite? For finite groups, we gave the following answer in \cite{ElmerKohlsSigmaDelta}:
\begin{Theorem}[{\cite[Theorem 1.1]{ElmerKohlsSigmaDelta}}]\label{thmdeltag}
Let $G$ be a finite group, $\kk$ an algebraically closed field of characteristic $p$, and $P$ a Sylow-$p$-subgroup of $G$. Then $\delta(G) = |P|$.
\end{Theorem} 
Thus, $\delta(G)$ is finite for all finite groups, and strictly greater than 1 if and only if $|G|$ is divisible by $p$. 
In this article we investigate $\delta(G)$ for infinite groups. In particular, we make and investigate the following conjecture:

\begin{Conj}\label{SsConj} 
Suppose $G$ is an infinite reductive group over a field of positive
characteristic whose connected component
$G^{0}$ is not a torus. Then $\delta(G) = \infty$.
\end{Conj}

The assumption on $G^{0}$ cannot be dropped, as the following
result, a generalisation of Theorem~\ref{thmdeltag} shows:

 \begin{Theorem}\label{deltaTorus}
 Let $G$ be a reductive linear algebraic group over a field  of
 characteristic $p > 0$ such that $G^0$ is a torus or trivial. Let $P$ be a
 Sylow-$p$-subgroup of the (finite) group $G/G^0$. Then $\delta(G)=|P|$.
\end{Theorem}

While we do not succeed in proving our conjecture, we are able to establish:

\begin{Theorem}\label{DeltaGL2} 
Let $\kk$ be a field of characteristic $p>0$. Suppose $G$ is a closed subgroup of
$\GL_2(\kk)$ containing an infinite unipotent subgroup. Then $\delta(G) = \infty$.
\end{Theorem}
 
In particular, $\delta(\SL_2(\kk))=\delta(\GL_2(\kk))=\delta(\Ga) = \infty$
(where $\Ga=(\kk,+)$ is the additive group of the ground field) in positive characteristics, supporting the
conjecture. 

In addition to $\delta(G)$, we study the closely related quantity $\sigma(G)$.  We shall say a subset $S \subseteq \kk[V]^G$ is a $\sigma$-\emph{set} if, for all $v \in V \setminus
\mathcal{N}_{G,V}$, there exists an $f\in S_{+}$  such that $f(v) \neq 0$. 
We shall call a subalgebra of $\kk[V]^G$ a $\sigma$-\emph{subalgebra} if it is
a $\sigma$-set. Then the quantities $\sigma(G,V)$ and $\sigma(G)$ are defined
along the same lines as $\delta(G,V)$ and $\delta(G)$. For a motivation of the
importance of this number we content ourselves here by saying that at least for
linearly reductive groups in characteristic zero, the knowledge of
$\sigma(G,V)$ gives upper bounds for the maximal degrees of generating sets
(for example in Derksen's famous bound \cite{DerksenPolynomial}), and refer
the reader to \cite{DomokosCziszter} and  \cite{ElmerKohlsSigmaDelta} for
more details and some elementary properties of this number.

In the latter paper, the authors investigated $\sigma(G)$ for finite groups
$G$, mainly for positive characteristic. In sections 4 and 5 of the present article we investigate $\sigma(G)$ for infinite linear algebraic groups. Our main results are as follows:

\begin{Theorem}\label{SigmaCharP} 
Let $G$ be a linear algebraic group over a field of characteristic $p > 0$. Then $\sigma(G)$ is finite if and only if $|G|$ is finite.
\end{Theorem}

\begin{Theorem}\label{SigmaChar0} 
Let $G$ be a linear algebraic group over a field of characteristic~$0$. Then if $\sigma(G)$ is finite, $G^{0}$ is unipotent, i.e. either $G$ is finite or $G^{0}$ is infinite unipotent.
\end{Theorem}

As reductive groups do not contain a non-trivial connected unipotent normal
subgroup, we get as an immediate corollary:

\begin{Corollary}
Let $G$ be a reductive group over a field of arbitrary
characteristic. Then $\sigma(G)$ is finite if and only if $G$ is finite.
\end{Corollary}

Somewhat surprisingly, for the (infinite) additive group $\Ga=(\kk,+)$ of a
field $\kk$ of characteristic zero, we will see $\sigma(\Ga)=2$. We do not
know whether $\sigma(G)$ is finite for all unipotent groups in characteristic zero. 

Another quantity associated with $\delta(G,V)$ and $\sigma(G,V)$, which has
attracted some attention in recent years, is $\beta_{\sep}(G,V)$.  It is
defined as follows: a subset $S \subseteq \kk[V]^G$ is called a
\emph{separating set} if, for any pair $v, w \in V$ such that there
exists $f \in \kk[V]^G$ with $f(v) \neq f(w)$, there exists $s \in S$ with
$s(v) \neq s(w)$. Now again, $\beta_{\sep}(G,V)$ and $\beta_{\sep}(G)$ are defined
along the same lines as $\sigma(G,V)$ and $\sigma(G)$.
Our point of view is that $\delta$- and $\sigma$-sets are ``zero-separating''
sets. This leads to the inequalities \cite[Proposition 1.4]{ElmerKohlsSigmaDelta} 
\[\delta(G,V) \leq \sigma(G,V) \leq \beta_{\sep}(G,V) \leq \beta(G,V)\]
for any linear algebraic group $G$ and $G$-module $V$, hence
\[
\delta(G) \leq \sigma(G) \leq \beta_{\sep}(G) \leq \beta(G),
\]  
where $\beta(G,V)$ and $\beta(G)$ are the classical local and global
Noether number. The second author and Kraft
 have shown \cite{KohlsKraft} that $\beta_{\sep}(G)$ is finite if and only if
$G$ is finite (indepedently of the characteristic of $\kk$). Some parts of
the proof of this result required some deep results from geometric invariant
theory. The results of our current paper allows one to replace these parts of the
proof by more elementary arguments, see section \ref{SigmaSection}.

\section{General results on the $\delta-$ number}\label{GeneralDelta}

In this section, we prove various general results on $\delta(G)$. For the convenience of the reader,
we present the proof of the following result of Nagata and Miyata in language consistent with this article.

\begin{Prop}[{Nagata and Miyata \cite[Proof of Theorem~1]{NagataMiyata}}]\label{NagataMiyata} 
Let $G$ be a linear algebraic group over a field $\kk$, and $V$  a
$G$-module. Suppose $v \in V^G$ and  $f \in \kk[V]_{+}^G$ is homogeneous such that $f(v)
\neq 0$. If the characteristic of $\kk$ does not divide the degree of $f$,
then there exists a homogeneous invariant $\tilde{f} \in \kk[V]_1^G$ of degree
one satisfying $\tilde{f}(v) \neq 0$. 
\end{Prop}

\begin{proof} 
Write $d:= \deg(f)$. Choose a basis $\{v=:v_0, v_1, \ldots, v_n\}$ of $V$ and
let $\{x_0,x_1,\ldots, x_n\}$ be the corresponding dual basis. Since $f(v)
\neq 0$, we can write $f = \sum_{i=0}^dx_0^{d-i}c_i $ with $c_i \in
\kk[x_1,x_2,\ldots, x_n ]_i$ for each $i = 0, \ldots, d$ and $c_0\in\kk\setminus\{0\}$. Without loss we assume $c_{0}=1$. Further, since $v \in V^G$, note that $\langle
x_1,x_2,\ldots, x_n \rangle$ is a $G$-invariant space, and we can write $g
\cdot x_0 = x_0 + y(g)$ with $y(g) \in \langle x_1,x_2,\ldots, x_n \rangle$
for each  $g \in G$. For any $g\in G$, we have
\begin{eqnarray*}
g \cdot f& =& (g \cdot x_0)^d + (g \cdot c_1)(g \cdot x_0)^{d-1} + (\text{terms
of $x_{0}$-degree}\le d-2)\\
&=&  (x_0 + y(g))^d + (g \cdot c_1)(x_0 + y(g))^{d-1} +  (\text{terms
of $x_{0}$-degree}\le d-2)\\
&=& x_0^d + (dy(g) + (g \cdot c_1))x_0^{d-1} + (\text{terms
of $x_{0}$-degree}\le d-2)=f,
\end{eqnarray*}
since $f$ is invariant. Comparing coefficients of $x^{d-1}_0$ tells us that for any $g \in G$ we have $c_1 = dy(g) + (g \cdot c_1)$.
By assumption, the degree $d$ is invertible in $\kk$, and we now set
$\tilde{f}:= x_0+d^{-1}c_1$. Notice that $\deg(\tilde{f}) = 1$, and for any $g \in G$ we have 
\[
g \cdot \tilde{f} = g \cdot x_0 + d^{-1}(g \cdot c_1) = x_0 + y(g) +
d^{-1}(c_1 - dy(g)) = \tilde{f},
\]
 so $\tilde{f} \in \kk[V]^G_1$. Furthermore $\tilde{f}(v) = x_0(v) + d^{-1}c_1(v) = x_0(v)=1 \neq 0$, completing the proof.
\end{proof}

\begin{Corollary}\label{DeltaCharZero} 
Let $G$ be a linear algebraic group  and $V$  a $G$-module. Then $\delta(G,V)$
equals either $0$ or $1$ or is divisible by the characteristic of $\kk$. In particular, if $\kk$ is a field of characteristic zero, then $\delta(G) = 1$.
\end{Corollary}

\begin{proof} 
Note firstly that if $V^{G}\subseteq\NGV$, then $\delta(G,V) =
0$. Otherwise, $\delta(G,V)\ge 1$. Applying the above proposition shows that
for any $\delta$-set $S$ consisting of homogeneous invariants, the set
$\kk[V]^{G}_{1}\cup \{f\in S\mid  \deg(f) \text{ divisible by the
  characteristic}\}$ is also a $\delta$-set. Finally, since $\delta(G,V) = 1$ when $V=\kk$ is the trival module, we must have $\delta(G) \geq 1$ for any linear algebraic group $G$.  
\end{proof}

The proof of the following result is a slight adaption of Nagata \cite[Lemma
3.1]{NagataAffine}, where it is shown that if $N$ is a closed normal subgroup
of $G$ such that $N$ and $G/N$ are reductive, then $G$ is reductive.

\begin{Prop}\label{DeltaNormalSubgroupFiniteIndex} 
Let $N$ be a closed normal subgroup of $G$ such that $G/N$ is reductive. Then
for any $G$-module $V$, we have
\[\delta(G,V) \leq \delta(N,V)\delta(G/N)\le \delta(N)\delta(G/N),\]
so in particular we have $\delta(G)\le\delta(N)\delta(G/N)$.
\end{Prop}

\begin{proof}
Take a point $v\in V^{G}\setminus\NGV$. As a $G$-invariant separating $v$ from zero is clearly
also an $N$-invariant, we see that $v \in V^N \setminus
\mathcal{N}_{N,V}$. Therefore there is a homogeneous $f_{0}\in\kk[V]^{N}$ of
positive degree $d\le\delta(N,V)$ satisfying $f_{0}(v)\ne 0$. Without
loss, we assume $f_{0}(v)=1$.
Note that  as $N$ is a normal subgroup of $G$, we have that $U:=\kk[V]_{d}^{N}$ is a $G$-module on which $N$ acts trivially, so it can be considered as a
$G/N$-module. Further we define $U_{0}:=\{f\in U\mid f(v)=0\}$. Note that
$U_{0}$ is a $G$-invariant subspace of $U$, since $v\in V^{G}$. As
$f_{0}\not\in U_{0}$, we have $U_{0}\ne U$. For any $f\in U$, we have
$f=(f-f(v)f_{0})+f(v)f_{0}$ with $f-f(v)f_{0}\in U_{0}$, hence $U=U_{0}\oplus
\kk f_{0}$ as a vector space. We can therefore define $\varphi\in U^{*}$ by
$\varphi(u_{0}+\lambda f_{0}):=\lambda$ for $u_{0}\in U_{0}$ and $\lambda\in
\kk$. It is easily seen that $\varphi$ is $G$-invariant. As mentioned, we can consider $U$ as a $G/N$-module, and then we
have $\varphi\in (U^{*})^{G/N}\setminus\{0\}$. By assumption, $G/N$ is a
reductive group, so there exists a homogeneous
$F\in\kk[U^{*}]^{G/N}_{d'}=S^{d'}(U)^{G/N}$ of some positive degree $d'\le\delta(G/N)$
such that $F(\varphi)\ne 0$. Let $\{f_{1},\ldots,f_{r}\}$ denote a basis of
$U_{0}$. Since $\varphi|_{U_{0}}=0$,  the fact that $F\in S^{d'}(\langle f_{0},f_{1},\ldots,f_{r}\rangle)^{G/N}$ such
that $F(\varphi)\ne 0$ implies $F=c\cdot f_{0}^{d'}+\tilde{F}$, where
$c\in\kk\setminus\{0\}$ and $\tilde{F}$ is an element of the ideal
$(f_{1},\ldots,f_{r})S(U)$. Note that as $U=\kk[V]^{N}_{d}$, there
is a canonical map $S^{d'}(U)^{G/N}\rightarrow \kk[V]^{G}_{dd'}$, so we can take $F$ as
an element of $\kk[V]^{G}_{dd'}$. Clearly, $F(v)=cf_{0}(v)^{d'}\ne 0$ as
$f_{i}(v)=0$ for $i=1,\ldots,r$ by the definition of $U_{0}$, showing
that $\delta(G,V)\le\delta(N,V)\delta(G/N)$.
\end{proof}

\begin{Corollary}\label{DeltaG0} 
Let $G$ be a linear algebraic group and let $G^0$ denote the connected component of the identity. Then we have
\[
\delta(G) \leq \delta(G/G^0)\delta(G^0).
\]
In particular, $\delta(G)$ is finite if $\delta(G^0)$ is finite.
\end{Corollary}

\begin{rem}\label{DeltaQuotient} 
If $N$ is a normal subgroup of $G$, then
$\delta(G/N) \leq \delta(G)$, since any $G/N$-module becomes a $G$-module via the map $G\rightarrow G/N$. 
\end{rem}

\begin{proof}[Proof of the Theorem \ref{deltaTorus}] 
As tori are
linearly reductive, $\delta(G^{0})=1$. Hence we get $\delta(G/G^{0})\le
\delta(G)\le\delta(G^{0})\delta(G/G^{0})=\delta(G/G^{0})$, so
$\delta(G)=\delta(G/G^{0})$. As $G/G^{0}$ is a finite group, the value of $\delta(G/G^0)$ is the size of a Sylow-$p$-subgroup by Theorem \ref{thmdeltag}.
\end{proof}

Theorem \ref{deltaTorus} shows that there are many examples of infinite groups
$G$ with finite $\delta(G)>1$; simply define $G = P \times T$ where $P$ is a
finite $p$-group and $T$ a nontrivial torus, then $\delta(G)=|P|$. For a more
interesting example, consider $G = \operatorname{O}_2(\kk)$ with $\kk$ an
algebraically closed field
of characteristic $2$. It is well known that $G \cong \kk^{*} \rtimes Z_{2}$, where $Z_2$ denotes the
cyclic group of order 2. Therefore $G^0 \cong \kk^*$ is a torus, and
$G/G^{0}\cong Z_{2}$. By Theorem \ref{deltaTorus}, $\delta( \operatorname{O}_2(\kk))=2$.

\section{The $\delta$-number for subgroups of $\GL_2(\kk)$}

The goal of this section is to prove Theorem \ref{DeltaGL2}. Throughout we assume $\kk$ is a field of characteristic $p>0$. We begin by introducing another number associated to a representation of a group, which is useful for finding lower bounds for both the $\delta$-number and $\sigma$-number. Let $G$ be a linear algebraic group and $V$ a $G$-module. Let $v \in V$. Then we set
\begin{equation*}
\epsilon(G,v): = \inf\{d\in\N_{>0}\mid
  \text{ there exists } f \in \kk[V]_d^G \text{ such that } f(v) \neq 0\},
\end{equation*}
where the infimum of an empty set is infinity. 
Notice that if $V^{G}\setminus\NGV\ne\emptyset$, then
\[
\delta(G,V)=\sup\{\epsilon(G,v)\mid v \in V^{G}\setminus\NGV\},
\]
and if $V\setminus\NGV\ne\emptyset$, then
\[
\quad\sigma(G,V)=\sup\{\epsilon(G,v)\mid v\in
  V\setminus\NGV\}.
\]
For a submodule $W \subseteq V$ we define
\begin{equation*} \epsilon(G,W,V) :=
  \inf\{\epsilon(G,v)\mid v \in W \setminus \mathcal{N}_{G,V}\},
\end{equation*}
and we set
\[
\epsilon(G,V):=\epsilon(G,V^{G},V) \quad\text{ and
}\quad\tau(G,V):=\epsilon(G,V,V).
\]

It is immdiately clear that for any linear algebraic group $G$ we have
$\delta(G,V) \geq \epsilon(G,V)$ if $V^{G}\setminus\NGV\ne\emptyset$ and
$\sigma(G,V) \geq \tau(G,V)$ if $V\setminus\NGV\ne\emptyset$ . In fact, we
have the following slightly stronger result, which we mainly use for $H$ a
\emph{finite} subgroup of $G$ (the second inequality is not used and only
stated for completeness):

\begin{Lemma}\label{RelativeEpsilon} 
Let $G$ be a linear algebraic group, $V$ a $G$-module and $H$ a subgroup of
$G$. Then
$\delta(G,V) \geq \epsilon(H,V)$ if $V^G \setminus \mathcal{N}_{G,V} \neq
\emptyset$  and $\sigma(G,V) \geq \tau(H,V)$ if $V\setminus\NGV\ne\emptyset$.
\end{Lemma}

\begin{proof} 
Choose a $v\in V^{G}\setminus\NGV$ such that
$\delta(G,V)=\epsilon(G,v)$. Clearly $v\in V^{H}\setminus\mathcal{N}_{H,V}$, hence
$\delta(G,V)=\epsilon(G,v)\ge\epsilon(H,v)\ge \epsilon(H,V^{H},V)=\epsilon(H,V)$. For the
second inequality, choose $v\in V\setminus\NGV$ such that
$\sigma(G,V)=\epsilon(G,v)$. As also $v\in V\setminus\mathcal{N}_{H,V}$, $\sigma(G,V)=\epsilon(G,v)\ge\epsilon(H,v)\ge\epsilon(H,V,V)=\tau(H,V)$.
\end{proof}

We believe a thorough investigation of the numbers $\epsilon(G,V)$ when $G$ is
a finite group may hold the key to proving Conjecture \ref{SsConj}. In order
to prove Proposition \ref{DeltaGL2}, we require only the corollary of the following lemma, whose proof is very similar to the proof of
\cite[Proposition~2.5]{ElmerKohlsSigmaDelta}, but the point of view is different. For any finite group $G$, let $V_{\reg,G}:=\kk G$ denote its
regular representation. 

\begin{Lemma}\label{EpsilonFree}  
Suppose $G$ is a finite group and $P$ a Sylow-$p$-subgroup of $G$. If
$V=V_{\reg,G}^{n}$ is a free $G$-module over $\kk$, then $\epsilon(G,v)=|P|$
for any $v\in V^{G}\setminus\{0\}$.
\end{Lemma} 

\begin{proof} 
For each $i=1,\ldots, n$ choose a permutation basis $\{v_{g,i}\mid g \in G\}$
of the $i$th summand (which is isomorphic to $V_{\reg,G}$), so that $\{v_{g,i}\mid g
\in G,\,\, i=1,\ldots,n\}$ is a basis of $V$. Let $\{x_{g,i}\mid g \in G,\,\,i = 1,\ldots, n\}$ be the basis dual to our
chosen basis of $V$, so that $\kk[V] = \kk[x_{g,i}: g \in G,\,\,
i=1,\ldots,n]$.
The fixed point space of the $i$th summand is spanned by $v_i:= \sum_{g
  \in G}v_{g,i}$, therefore $V^{G}=\langle v_{1},\ldots,v_{n}\rangle$. 
For a point $v=\sum_{i=1}^{n}\lambda_{i}v_{i}\in V^{G}\setminus\{0\}$ with scalars
$\lambda_{i}\in\kk$ (not all of them zero), we will show
$\epsilon(G,v)=|P|$. We show first $\epsilon(G,v)\ge|P|$, i.e.
$\deg(f)\ge|P|$ for any homogeneous $f \in \kk[V]_+^G$ such that $f(v) \neq
0$.  Since $V$ is a permutation module,  such an $f$ is a linear combination of orbit sums of monomials
\[O_G(m):= \sum_{m' \in G\cdot m} m',\] where $m$ is a monomial in
$\kk[V]_+$. It follows that there exists a monomial $m \in \kk[V]_+$, whose
degree is the same as $\deg(f)$, such that $O_G(m)(v) \neq 0$. Now if $m' \in G\cdot m$ then $m' = g\cdot m$ for some $g \in G$, and $m'(v) = (g \cdot m) (v) = m(g^{-1}v) = m(v)$ since $v \in V^G$. Therefore 
\[ 
O_G(m)(v) = |G\cdot m| m(v)=(G:\Stab_{G}(m))m(v)\neq 0.
\]
This implies that $\Stab_{G}(m)$ contains a Sylow-$p$-subgroup of $G$, which
without loss we can assume to be $P$.  Therefore, if
$x_{g,i}$ is any variable dividing $m$, then $m$ is also divisible by
$x_{g'g,i}$ for every $g' \in P$. In particular, since $m$ is not constant, we
obtain $\deg(f)=\deg(m) \geq |P|$ as required. Secondly, choose an $i$ such
that $\lambda_{i}\ne 0$ and define $m:=\prod_{g\in P}x_{g,i}$. Then $O_{G}(m)$
is an invariant of degree $|P|$ satisfying
\[
O_{G}(m)(v)=(G:\Stab_{G}(m))m(v)=(G:P)\lambda_{i}^{|P|}\ne 0,
\]
showing $\epsilon(G,v)\le|P|$.
\end{proof}

\begin{Corollary}\label{EpsilonProjective}  
Suppose $G$ is a finite group and $P$ a Sylow-$p$-subgroup of $G$. If
$U$ is a projective $G$-module over $\kk$, then $\epsilon(G,u)=|P|$
for any $u\in U^{G}\setminus\{0\}$. In particular, $\epsilon(G,U)=|P|$ if $U^{G}\ne\{0\}$.
\end{Corollary}

\begin{proof}
By definition, $U$ is a direct summand of a free module, i.e. there exists (up to
isomorphism) a
decomposition $U\oplus W=V=V_{\reg,G}^{n}$ of some free module $V$ into $U$
and a $G$-module complement $W$. Take $u\in U^{G}\setminus\{0\}$. As $u\in
V^{G}\setminus\{0\}$, by the previous lemma there is  an $f\in\kk[V]_{|P|}^{G}$
satisfying $f(u)\ne0$. As $f|_{U}\in \kk[U]^{G}_{|P|}$ satisfies
$f|_{U}(u)=f(u)\ne 0$, we have $\epsilon(G,u)\le|P|$. On the other hand, as we have
an algebra-inclusion $\kk[U]\subseteq \kk[V]$, any homogeneous
$f\in\kk[U]^{G}_{+}$ satisfying $f(u)\ne 0$ can be considered as an element of
$\kk[V]^{G}$, hence $\deg(f)\ge |P|$ by the same lemma, so $\epsilon(G,u)\ge|P|$.
\end{proof}

It is worth recalling that for a $p$-group, ``projective'' and ``free'' means the
same for a module. It is now clear how our proof should proceed - we need to find a
sequence of finite ($p$-)subgroups $H$ of $G = \GL_2(\kk)$ and $G$-modules $V$
which become projective (free) on restriction to $H$. The following result provides a good source of such modules.

\begin{Prop}\label{ProjectiveModules} Let $p>0$ be a prime and let $\kk$ be an
  algebraically closed field of characteristic $p$. Let $G_n =
  (\mathbb{F}_{p^n},+)$ be the additive group of the finite subfield
  $\FF_{p^{n}}$ of $\kk$. Let $V$ be the $G_n$-module spanned by vectors $X$
  and $Y$ such that the action $*$ of $G_n$ on $V$ is given by
\begin{eqnarray*}
{t} *  X &=& X \\
\text{ and }\quad {t} * Y &=& Y + {t} X\quad\quad\text{ for all }t\in G_{n}.
\end{eqnarray*}
Then $S^{p^n-1}(V)$ is isomorphic to the regular representation of $G_{n}$.
\end{Prop} 

\begin{proof} 
We will show that $S:=\{{t} * Y^{p^n-1}\mid {t} \in G_n\}$ is a basis of
$S^{p^n-1}(V)$, which clearly implies $S^{p^{n}-1}(V)\cong V_{\reg,G_n}$.  As $|G_{n}|=p^{n}$
equals the dimension of $S^{p^n-1}(V)$, it is enough to show that the
$p^{n}\times p^{n}$ matrix $A$ with columns formed by the coordinate vectors of the elements
${t} * Y^{p^n-1},\, t\in G_{n}$ with respect to the standard basis
$\{Y^{p^n-1-i}X^{i}\mid i \in \{0,\ldots,p^{n}-1\}\}$ of $S^{p^{n}-1}(V)$ has a nonzero determinant.
Using the binomial theorem and Lemma \ref{NumberTheory} we compute
\begin{eqnarray*}
{t} * Y^{p^n-1} &= &(Y+ {t} X)^{p^n-1} = \sum_{i = 0}^{p^n-1} {p^n-1 \choose
  i} Y^{p^n-1-i}(t X)^{i}\\
&\stackrel{\text{L. } \ref{NumberTheory}}{=} &\sum_{i = 0}^{p^n-1} (-1)^{i}Y^{p^n-1-i}(tX)^{i} = \sum_{i = 0}^{p^n-1}
(-{t})^{i}Y^{p^n-1-i}X^{i}.
\end{eqnarray*}
Thus, $A=((-t)^{i})_{i\in\{0,\ldots,p^{n}-1\},t\in G_{n}}\in \kk^{p^{n}\times
  p^{n}}$, where we enumerated the $p^{n}$ columns of $A$ by the set $G_{n}$
-- which is harmless as the order of the columns only affects the sign of the determinant of
$A$. Note that $A$ is the $p^{n}\times p^{n}$ Vandermonde matrix of the
$p^{n}$ \emph{different} elements of $-G_{n}(=G_{n})$, hence $\det(A)\ne 0$,
which proves the claim.
\end{proof}

In the preceding proof we used the following number-theoretic lemma, of which
we provide a proof for the convenience of the reader.
\begin{Lemma}\label{NumberTheory} Let $p$ be a prime number and $0\le k \leq
  p^n-1$. Then 
\[
{p^n-1 \choose k}\equiv (-1)^{k}\mod p.
\]
\end{Lemma}

\begin{proof}
We have ${p^n-1 \choose k}=\prod_{m=1}^{k}\frac{p^{n}-m}{m}$. We show that
the reduced fraction of each factor has a denominator coprime to $p$, and
equals $-1+p\Z$ if computed in the field $\ZZ/p\ZZ$. For this sake, for  $1\le m\le
k\le p^{n}-1$, write $m=p^{r}s$ where $s$ and $p$ are coprime. Then $r<n$, and
$\frac{p^{n}-m}{m}=\frac{p^{n}-p^{r}s}{p^{r}s}=\frac{p^{n-r}-s}{s}$. In the
field $\ZZ/p\ZZ$, the last fraction equals $-1$.
\end{proof}

Having set up all the necessary machinery, we are now in a position to prove
Theorem \ref{DeltaGL2}. Let $G$ be a subgroup of $\GL_2(\kk)$
containing an infinite unipotent subgroup $U$. As $U$ is conjugate in
$\GL_{2}(\kk)$ to the subgroup of unipotent upper triangular $2\times 2$ matrices
(see \cite[Corollary 17.5]{Humphreys}), we can replace $G$ by a conjugate subgroup
and assume $U=\left\{ u_{t}\mid t \in\kk\right\}$, where
  $u_{t}=\tiny\left(\begin{array}{cc}1&t\\&1\end{array}\right)\in G$.
Note that $U$ is ismorphic to the additive group of the ground field $\Ga=(\kk,+)$. Let
$V$ denote the restriction of the natural $2$-dimensional $\GL_2(\kk)$-module to
$G$. We may choose a basis  $\{X,Y\}$ of $V$ such that 
\[
 u_{t} * X = X \quad\text{ and }\quad u_{t}* Y = Y+tX \quad \text{ for all } t\in\Ga.
\]  
Theorem \ref{DeltaGL2} follows immediately from the following:
\begin{Prop} 
For any integer $n$ set $V_n: =
\Hom_{\kk}(S^{p^n-1}(V),S^{p^n-1}(V))$. Then $\delta(G,V_n) \geq p^n.$ 
\end{Prop}

\begin{proof} 
First note that $V_n^G \setminus \mathcal{N}_{G,V_n} \neq \emptyset$: to see
this consider the identity homomorphism $\id: S^{p^n-1}(V) \rightarrow
S^{p^n-1}(V)$, which is an element of $V_n^G$. The determinant map
$\det: V_n \rightarrow \kk$ is an element of $\kk[V_n]^G$, and $\det(\id) =
1\ne 0$,
so $\id \in V_{n}^{G}\setminus \mathcal{N}_{G,V_n}$. Therefore we can apply
Lemma \ref{RelativeEpsilon} to $G$ and its finite subgroup $U_n:=\{u_{t}\mid
t\in \FF_{p^{n}}\}$, hence $\delta(G,V_{n})\ge \epsilon(U_{n},V_{n})$. Note
that $U_{n}\cong(\FF_{p^{n}},+)$. By
Proposition \ref{ProjectiveModules}, $S^{p^n-1}(V)$ is a free
$U_n$-module. Recall that tensoring a free/projective module with any
other module yields again a free/projective module, hence
$V_{n}=\Hom_{\kk}(S^{p^n-1}(V),S^{p^n-1}(V))\cong S^{p^n-1}(V)\otimes (S^{p^n-1}(V))^{*}$
 is also a free $U_n$-module. Using Corollary \ref{EpsilonProjective} we obtain
\[
\delta(G,V_n)\ge\epsilon(U_n,V_n)= |U_{n}|=p^{n} 
\] 
as required.
\end{proof}

We record the following observation for later use:
\begin{Corollary}\label{DeltaUnipotent} 
Let $G$ be an infinite connected unipotent algebraic group over an
algebraically closed field of positive characteristic. Then $\delta(G) = \infty$.
\end{Corollary}

\begin{proof} 
It is well-known that such a group $G$ contains a closed normal subgroup $N$
such that $G/N \cong\Ga$. We can embed $\Ga$ in $\GL_2(\kk)$ as above. Now using Remark \ref{DeltaQuotient} and Theorem \ref{DeltaGL2}, we have $\delta(G) \geq \delta(G/N)=\delta(\Ga) = \infty.$
\end{proof}

Combining Theorem \ref{DeltaGL2} and Proposition
\ref{DeltaNormalSubgroupFiniteIndex} leads to more examples of groups with
infinite $\delta$-value: Whenever $\delta(G)=\infty$ and $N$ is a closed normal
subgroup of $G$ such that $G/N$ is reductive, either $\delta(N)=\infty$ or
$\delta(G/N)=\infty$. 

\begin{eg}
Take $G=\GL_{2}(\kk)$ and consider its centre $Z(G)=\{a I_{2}\mid a\in\kk\setminus\{0\}\}$. As a
torus, $Z(G)$ is linearly reductive, hence $\delta(Z(G))=1$. Therefore,
$\delta(\operatorname{PGL}_{2}(\kk))=\delta(G/Z(G))=\infty$. Note also that
$\delta(\operatorname{PSL}_{2}(\kk))=\infty$, because over an algebraically
closed field, we have $\operatorname{PSL}_{2}(\kk)\cong\operatorname{PGL}_{2}(\kk)$.
\end{eg}

\section{The $\sigma$-number of infinite groups}\label{SigmaSection}

In this section we prove Theorems \ref{SigmaCharP} and \ref{SigmaChar0}. Some of the groundwork was done in \cite{ElmerKohlsSigmaDelta}. In particular, we recall the following result:

\begin{Prop}\label{sigmaGG0}\cite[Corollary~3.13]{ElmerKohlsSigmaDelta}
Let $G$ be a linear algebraic group, with $G^0$ the connected component of $G$ containing the identity. We have the inequalities
\[
\sigma(G^{0})\le\sigma(G)\le(G:G^{0})\sigma(G^{0}).
\]
In particular, $\sigma(G)$ and $\sigma(G^{0})$ are either both
finite or infinite.
\end{Prop}

The following proposition is key to the proofs:
\begin{Prop}\label{SigmaSemisimple}
Let $G$ be a linear algebraic group over a field $\kk$ of arbitrary
characteristic. Suppose $G$ contains a non-trivial torus. Then $\sigma(G) = \infty$.
\end{Prop}

\begin{proof}  
We exhibit a sequence of $G$-modules $\{U_m\mid m \in
  \mathbb{N}\}$ such that $\sigma(G,U_m) \ge m+1$ for all $m\in\N$. By assumption, $G$ contains a
  subgroup  $T \cong \kk^*$, so there is an isomorphism $\kk^{*}\rightarrow
  T$, $t\mapsto a_{t}$. As a linear algebraic group, $G$ can be
  considered as a closed subgroup of some $\GL_{n+1}(\kk)$, and then $V=\kk^{n+1}$
  becomes a faithful $G$-module. We
  can choose a basis $\{v_0,v_1,\ldots, v_n\}$ of $V$ on which $T$ acts
  diagonally, and as $T$ acts faithfully, it acts non-trivially on at least
  one basis vector, say $v_{0}$. Therefore, for some $r\in\Z\setminus\{0\}$,
  we have $a_{t}*v_{0}=t^{r}v_{0}$ for all $t\in\kk^{*}$. Write $\{y_0,y_1,\ldots, y_n\}$ for the basis of $V^*$ dual to
$\{v_0,v_1,\ldots, v_n\}$. A basis for $S^m(V^*)$ is then given by the set of
monomials 
\[
\left\{{\bf
  y}^{{\bf e}}:=\prod_{i=0}^{n} y_i^{e_i}\in S^m(V^*) \,\,\mid\,\, {\bf e}\in\N_{0}^{n+1},\,\, |{\bf
  e}|:=\sum_{i=0}^{n}e_{i}=m\right\}
\]
 of
degree $m$ in this basis. Let further \[
\{Z_{\bf e}\in S^{m}(V^{*})^{*}\mid
{\bf e}\in\N_{0}^{n+1},\,\,|{\bf e}|= m\}
\]
 denote the corresponding dual basis  of $S^m(V^*)^*$  i.e. $Z_{{\bf e}}(y^{\bf
  e'})=\delta_{{\bf e},{\bf e'}}$ (the Kronecker-delta).  Now we set $U_m:= V \oplus S^m(V^*)$.
We may identify $\kk[U_m]$ with 
\[
S(U_m^*) = S(V^* \oplus S^m(V^*)^*) = \kk[y_0,y_1,\ldots, y_n]\left[ Z_{\bf
    e}: {\bf e}\in\N_{0}^{n+1},\,\, |{\bf e}|=m\right].
\] 
Consider the point $v:= v_0+y_0^m \in U_m$. We claim that $v \not \in
\mathcal{N}_{G,U_m}$, and we will show that $\epsilon(G,v)=m+1$. As a
consequence, $\sigma(G,U_{m})\ge\epsilon(G,v)=m+1$, finishing the proof.
To see this, we define the polynomial 
\[
f:= \sum_{{\bf e}\in\N_0^{n+1},\,|{\bf e}|=m} {\bf y}^{\bf e}Z_{\bf e} \in
\kk[U_m],
\]
which can be interpreted as the identity map $\id: S^{m}(V^{*})\rightarrow
S^{m}(V^{*})$, and is hence an invariant, i.e. $f\in \kk[U_{m}]^{G}$. Note
that here we used the isomorphism
\[
\Hom_{\kk}(S^m(V^*),S^m(V^*))\cong S^m(V^*) \otimes S^m(V^*)^*
\]
and that 
\[
\kk[U_{m}]\cong S(V^* \oplus S^m(V^*)^*) \cong S(V^*) \otimes S(S^m(V^*)^*)
\]
contains a direct summand isomorphic to $S^m(V^*) \otimes S^m(V^*)^*$. Clearly
$f(v) = 1 \neq 0$, which shows that $v\not \in \mathcal{N}_{G,U_m}$.
Furthermore we have $\deg(f) = m+1$, so we have $\epsilon(G,v) \leq m+1$. It
remains to show that $\epsilon(G,v) \geq m+1$. Suppose a homogeneous $f' \in \kk[U_m]_{+}^G$
also satisfies $f'(v) \neq 0$; we will show that $\deg(f') \geq m+1$. Observe
that a fortiori we have $f' \in \kk[U_m]_{+}^T$. Therefore $f'$ can be written as
a sum of $T$ invariant monomials, so in particular there exists a
$T$-invariant monomial $h$ (of the same degree as $f'$) satisfying $h(v) \neq
0$. As $v=v_0+y_0^m$, the only variables that can appear in $h$ are those dual
to $v_{0}$ and $y_{0}^{m}$, i.e. the variables $y_{0}$ and $Z_{{\bf e}_{0}}$ with ${\bf
  e}_{0}:=(m,0,0,\ldots,0)$. We thus have $h=y_{0}^{k}Z_{{\bf e}_{0}}^{l}$ with
$k,l\in\N_{0}$, and $\deg(h)=k+l>0$. On the other hand, since $h \in \kk[U_m]^T$ we have 
\begin{eqnarray*}
y_{0}^{k}Z_{{\bf e}_{0}}^{l}&=&h = a_{t} * h = (a_{t} * y_0)^{k}(a_{t} * Z_{{\bf e}_{0}})^{l}\\& =& 
(t^{-r}y_{0})^{k}(t^{mr}  Z_{{\bf e}_{0}})^{l}=t^{mrl-kr}\cdot y_{0}^{k}Z_{{\bf
    e}_{0}}^{l} \quad \text{ for all }t\in\kk^{*},
\end{eqnarray*}
i.e. $r(ml-k)=0$. Since $r\ne 0$ and $k+l>0$ it must be the case that $k=ml\ge m$
and $l\ge 1$. Therefore $\deg(f') = \deg(h) = ml+l\ge m+1$ as required.
\end{proof}

\begin{Corollary} 
Suppose $G$ is a linear algebraic group such that $\sigma(G)$ is finite. Then
$G^{0}$ is unipotent, i.e. either $G$ is finite or $G^{0}$ is infinite unipotent.
\end{Corollary}

\begin{proof} 
If
$\sigma(G)$ is finite, $\sigma(G^{0})$ is finite by Proposition \ref{sigmaGG0}. It follows from Proposition \ref{SigmaSemisimple} that
$G^{0}$ does not contain any non-trivial torus, i.e. the rank of the connected
group $G^{0}$ (the
dimension of a maximal torus) is zero, hence $G^{0}$ is unipotent by
\cite[Exercise 21.4.1]{Humphreys}.
\end{proof}

Specialising to the case of $\kk$ a field of characteristic zero, this
completes the proof of Theorem~\ref{SigmaChar0}. To finish the proof of
Theorem~\ref{SigmaCharP}, it remains to show that over a field of positive
characterstic, if $G^{0}$ is infinite unipotent, we have
$\sigma(G)=\infty$. This follows from $\sigma(G^{0})\le\sigma(G)$ (Proposition \ref{sigmaGG0}), the inequality $\delta(G^{0})\le\sigma(G^{0})$
and from $\delta(G^{0})=\infty$ (Corollary \ref{DeltaUnipotent}). The following proposition, which provides
some examples of their own interest, gives a more direct proof that
$\delta(\Ga)=\sigma(\Ga)=\infty$ for a field of positive characteristic.
Additionally, it gives another proof of $\beta_{\sep}(\Ga)=\infty$ for such a field,
which is also shown in \cite[Proposition 4]{KohlsKraft}, see also the
following remark for more details. As before, it follows
$\delta(G)=\sigma(G)=\infty$ for any infinite unipotent connected group, via  a normal subgroup $N$ such that $G/N\cong \Ga$.
We want to mention that $\Ga$-modules of the type as in the proposition are
also investigated in \cite{Fauntleroy,TanimotoPolynomiality}. The generators of the considered
invariant ring would also follow from the latter paper, but we give a 
self-contained argument. 

\begin{Prop}
Assume $\kk$ is a field of characteristic $p>0$, and let $V_{n}=\kk^{3}$
($n\ge 1$) be the
$\Ga=(\kk,+)$-module given by  the representation
\[
\Ga\mapsto \GL_{3}(\kk),\quad t\mapsto\left(\begin{array}{ccc}1&0&0\\-t&1&0\\-t^{p^{n}}&0&1\end{array}\right).
\]
If we write $\kk[V_{n}]=\kk[x_{0},x_{1},x_{2}]$, then we have
\[
\kk[V_{n}]^{\Ga}=\kk[x_{0},x_{2}x_{0}^{p^{n}-1}-x_{1}^{p^{n}}] \quad \text{
  and } \quad \delta(\Ga,V_{n})=\sigma(\Ga,V_{n})=p^{n}.
\]
Consequently, $\delta(\Ga)=\sigma(\Ga)=\infty$.
\end{Prop}

\begin{proof}
The action $*$ of $\Ga$ on $\kk[V_{n}]$ is given by
\[
t*f(x_{0},x_{1},x_{2})=f(x_{0},x_{1}+tx_{0},x_{2}+t^{p^{n}}x_{0}) \quad \text{
for } t\in\Ga,\,\,\, f(x_{0},x_{1},x_{2})\in\kk[V_{n}].
\]
If $f$ is an invariant, the equation $t*f=f$ for all $t\in\Ga$ implies that
for an additional independent variable $t$, the equation
\[
f(x_{0},x_{1},x_{2})=f(x_{0},x_{1}+tx_{0},x_{2}+t^{p^{n}}x_{0})
\]
holds in the polynomial ring $\kk[V_{n}][t]$. Substituting $t:=-\frac{x_{1}}{x_{0}}$ leads to
\begin{eqnarray}\label{fx0x1x2}
f(x_{0},x_{1},x_{2})=f\left(x_{0},0,x_{2}-\frac{x_{1}^{p^{n}}}{x_{0}^{p^{n}}}x_{0}\right)=f\left(x_{0},0,\frac{x_{2}x_{0}^{p^{n}-1}-x_{1}^{p^{n}}}{x_{0}^{p^{n-1}}}\right).
\end{eqnarray}
We have to show
that
$\kk[V_{n}]^{\Ga}\subseteq\kk[x_{0},x_{2}x_{0}^{p^{n}-1}-x_{1}^{p^{n}}]$, as
the reverse inclusion is checked immediately. For
an $f\in \kk[V_{n}]^{\Ga}$, write
$f=\sum_{k=0}^{m}a_{k}(x_{0},x_{1})x_{2}^{k}$ with polynomials
$a_{k}\in\kk[x_{0},x_{1}]$. Equation \eqref{fx0x1x2} implies
\begin{eqnarray}\label{feqSum}
f=\sum_{k=0}^{m}a_{k}(x_{0},0)\left(\frac{x_{2}x_{0}^{p^{n}-1}-x_{1}^{p^{n}}}{x_{0}^{p^{n}-1}}\right)^{k}=\sum_{k=0}^{m}\frac{b_{k}(x_{0})}{(x_{0}^{p^{n}-1})^{k}}(x_{2}x_{0}^{p^{n}-1}-x_{1}^{p^{n}})^{k},
\end{eqnarray}
with polynomials $b_{k}(x_{0}):=a_{k}(x_{0},0)\in\kk[x_{0}]$. Substituting $x_{2}:=0$ leads to
\[
f(x_{0},x_{1},0)=\sum_{k=0}^{m}\frac{b_{k}(x_{0})}{(x_{0}^{p^{n}-1})^{k}}(-x_{1}^{p^{n}})^{k}\in\kk[x_{0},x_{1}],
\]
which implies that $c_{k}(x_{0}):=\frac{b_{k}(x_{0})}{(x_{0}^{p^{n}-1})^{k}}$ is
actually a polynomial, i.e. an element of $\kk[x_{0}]$. Resubstituting in
\eqref{feqSum} implies
\[
f=\sum_{k=0}^{m}c_{k}(x_{0})(x_{2}x_{0}^{p^{n}-1}-x_{1}^{p^{n}})^{k}\in\kk[x_{0},x_{2}x_{0}^{p^{n}-1}-x_{1}^{p^{n}}],
\]
 as desired.  It follows that $\sigma(\Ga,V_{n})\le p^{n}$ and $\mathcal{N}_{\Ga,V_{n}}=\{(0,0,a_{2})\in
 V_{n}\mid a_{2}\in\kk\}$, and clearly we have $V_{n}^{\Ga}=\{(0,a_{1},a_{2})\in
 V_{n}\mid a_{1},a_{2}\in\kk\}$.
Now the point $v:=(0,1,0)\in V_{n}^{\Ga}\setminus\mathcal{N}_{\Ga,V_{n}} $ satisfies $x_{0}(v)=0$ and
 $(x_{2}x_{0}^{p^{n}-1}-x_{1}^{p^{n}})(v)=-1$, which shows $\delta(\Ga,V_{n})=\sigma(\Ga,V_{n})=p^{n}$.
\end{proof}

\begin{rem} 
Theorems \ref{SigmaCharP} and \ref{SigmaChar0}  were proved by
  ``elementary'' means, in the sense that we did not use any geometric
  invariant theory. We can use these results to give an elementary proof of
  \cite[Theorem~A]{KohlsKraft}, which states that $\beta_{\sep}(G)$ is finite
  if and only if $G$ is finite. That $\beta_{\sep}(G)$ is finite for a finite
  group $G$ is well known (see \cite[Corollary~3.9.14]{DerksenKemper}) so it
  remains to prove the converse. Suppose $\beta_{\sep}(G)$ is finite. The
  inequality $\sigma(G)\le\beta_{\sep}(G)$ implies in particular $\sigma(G)$
  is finite, so if $\kk$ has characteristic $p>0$ we are done by Theorem \ref{SigmaCharP}. Otherwise we
  conclude that $G^0$ is unipotent from Theorem \ref{SigmaChar0}. 
Now the results $\beta_{\sep}(\Ga)=\infty$ and $\beta_{\sep}(G^{0})\le
\beta_{\sep}(G)$, which are both proven elementarily in \cite[Proposition~5 and
Theorem B]{KohlsKraft}, imply that $\beta_{\sep}(G)=\infty$ when $G^{0}$ is an infinite unipotent group. Hence, if $\beta_{\sep}(G)<\infty$, $G^{0}$ and $G$ are finite.
\end{rem}

We do not know very much about $\sigma(G)$ when $G$ is an infinite unipotent
group over a field of characteristic zero. Unlike $\beta_{\sep}(G)$, it is not
always infinite, as the following surprising result shows:

\begin{Prop}
Assume $\kk$ is a field of characteristic $0$. Then $\sigma(\Ga)=2$.
\end{Prop}

\begin{proof}
In \cite[Section 3]{ElmerKohls}, we give for any $\Ga$-module $V$ an
explicit set of invariants of degree at most $2$ that cuts out the nullcone. It follows that $\sigma(\Ga) = 2$.
\end{proof}

We conclude with an example which shows that $\sigma(\mathbb{G}_a \times \mathbb{G}_a) \geq 3$.

\begin{eg}
Let $\kk$ be an algebraically closed field of characteristic zero and
$V=\kk^4$. Consider an action of $G:=\mathbb{G}_a \times \mathbb{G}_a$ defined
as follows: $(s,t) \in \kk \times \kk$ acts on $V$ as multiplication by the matrix
\[ \left( \begin{array}{cccc}
1 & 0 & 0 & 0 \\
-s & 1 & 0 & 0 \\
\frac12s^2 - t & -s & 1 & 0 \\
-\frac16s^3+st & \frac12s^2-t & -s & 1
\end{array} \right).\]

Let $\{x_0,x_1,x_2,x_3\}$ denote the basis of $V^*$ dual to the standard basis of $V$. Then we claim that the ring of invariants $\kk[V]^G$ is generated by the invariants $x_0$ and $f:= x_1^3-3x_0x_1x_2+3x_0^2x_3$. Under this assumption we have that the point $v = (0,1,0,0) \in V$ is not contained in the nullcone, since $f(v)=1 \neq 0$, and is not separated from zero by any invariant of degree less than 3, which shows that $\sigma(G,V)=3$ and hence $\sigma(G) \geq 3$.

To prove the claim, consider the subgroup $H:=\{(0,t)\in G \mid t \in \kk\}$ of
$G$. The action of $H$ on $\kk[V]$ is given by 
\begin{eqnarray*}
(0,t)*x_0 &=& x_0\\ 
(0,t)*x_1 &=& x_1\\
(0,t)*x_2 &=& x_2+tx_0\\
(0,t)*x_3 &=& x_3+tx_1\quad\quad\text{ for all } t\in\kk.
\end{eqnarray*}
This $\Ga$-action corresponds to the direct sum of two copies of the natural representation of
$\Ga$, and the invariant ring is well known to be given by $\kk[V]^H =
\kk[x_0,x_1,x_0x_3-x_2x_1]$. Crucially, this is a polynomial ring in three
variables. Now $\kk[V]^G = \kk[x_0,x_1,x_0x_3-x_2x_1]^{G/H}$ is isomorphic to
the ring of invariants of a non-linear action of $\mathbb{G}_a$ on a
polynomial ring in three variables; by a theorem of Miyanishi (see
\cite[Theorem~5.1]{FreudenburgBook}) this ring of invariants is again
polynomial, with two generators. Therefore, $\kvg$ is a graded polynomial ring
with two generators. One may readily check that $x_0$ is the only
invariant of degree one, and as $f$ is an invariant of smallest possible
degree not contained in $\kk[x_{0}]$, we see that $\kk[V]^G=\kk[x_0,f]$ as claimed.
\end{eg}

\begin{ack}
This paper was prepared during visits of the second author to the University of Aberdeen, and of the first author to T.U. M\"unchen. The first of these visits was supported by the Edinburgh Mathematical Society's Research Support Fund. We want to thank Gregor Kemper and the Edinburgh Mathematical Society for making these visits possible. 
\end{ack}

\bibliographystyle{plain}
\bibliography{MyBib}

\end{document}